\documentclass[11pt,a4paper]{amsart}
\usepackage[margin=3.3cm]{geometry}
\usepackage{amssymb}
\usepackage{graphicx} 
\usepackage[mathscr]{euscript}
\usepackage{enumerate}
\usepackage{xspace}
\usepackage{color}
\usepackage{enumerate}
\usepackage{enumitem}
\usepackage{amsthm}
\usepackage{dsfont}
\usepackage{bbm}
\usepackage[colorlinks=true, linkcolor = blue, citecolor = blue]{hyperref}
\usepackage[numbers]{natbib}
\usepackage[backgroundcolor=white,bordercolor=red]{todonotes}
\DeclareMathAlphabet{\mathpzc}{OT1}{pzc}{m}{it}
\usepackage[nameinlink]{cleveref}
\begin{document}

\theoremstyle{plain}

\newtheorem{theorem}{Theorem}[section]
\newtheorem{lemma}[theorem]{Lemma}
\newtheorem{proposition}[theorem]{Proposition}
\newtheorem{corollary}[theorem]{Corollary}
\newtheorem{definition}[theorem]{Definition}
\newtheorem{Ass}[theorem]{Assumption}
\theoremstyle{definition}
\newtheorem{remark}[theorem]{Remark}
\newtheorem{SA}[theorem]{Standing Assumption}
\newtheorem{example}[theorem]{Example}
\newtheorem*{discussion}{Discussion}

\renewcommand{\chapterautorefname}{Chapter} 
\renewcommand{\sectionautorefname}{Section} 

\crefname{lemma}{lemma}{lemmas}
\Crefname{lemma}{Lemma}{Lemmata}
\crefname{corollary}{corollary}{corollaries}
\Crefname{corollary}{Corollary}{Corollaries}

\newcommand{\Y}{\mathsf{Y}}
\newcommand{\s}{\mathfrak{s}}
\newcommand{\m}{\mathfrak{m}}
\def\stackrelboth#1#2#3{\mathrel{\mathop{#2}\limits^{#1}_{#3}}}

\newcommand\myrot[1]{\mathrel{\rotatebox[origin=c]{#1}{$\Longrightarrow$}}}
\newcommand\NEarrow{\myrot{45}}
\newcommand\SEarrow{\myrot{-45}}

\renewcommand{\theequation}{\thesection.\arabic{equation}}
\numberwithin{equation}{section}

\newcommand{\C}{\mathbb{C}}
\newcommand{\on}{\operatorname}
\newcommand{\1}{\mathds{1}}
\renewcommand{\epsilon}{\varepsilon}
\newcommand{\X}{\mathsf{X}}
\newcommand{\Z}{\mathsf{Z}}
\newcommand{\E}{{\mathds{E}}}
\renewcommand{\P}{\mathds{P}}
\renewcommand{\t}{\tau} 
\newcommand{\B}{\mathsf{X}}
\newcommand{\Q}{\mathds{W}}
\newcommand{\T}{{\tau_D}} 
\renewcommand{\S}{{\tau_D}} 
\newcommand{\M}{\mathbf{M}}
\renewcommand{\i}{\zeta}
\newcommand{\z}{\gamma} 
\renewcommand{\u}{\mathfrak{u}}
\renewcommand{\v}{\mathfrak{v}}
\newcommand{\I}{o} 
\renewcommand{\emptyset}{\varnothing}
\newcommand{\g}{\mathfrak{g}}

\title[Feller--Dynkin and Martingale Property of Diffusions]{On the Feller--Dynkin and the Martingale Property of One-Dimensional Diffusions} 
\author[D. Criens]{David Criens}
\address{D. Criens - Albert-Ludwigs University of Freiburg, Ernst-Zermelo-Str. 1, 79104 Freiburg, Germany}
\email{david.criens@stochastik.uni-freiburg.de}

\keywords{Diffusion, Markov Process, Martingale, Feller Process, Cauchy Problem, Speed Measure, Scale Function, Irregular Points\vspace{1ex}}


\thanks{Financial support from the DFG project No. SCHM 2160/15-1 is gratefully acknowledged. \\}

\date{\today}
\maketitle

\frenchspacing
\pagestyle{myheadings}

\begin{abstract}
We show that a one-dimensional regular continuous Markov process \(\X\) with scale function \(\s\) is a Feller--Dynkin process precisely if the space transformed process \(\s (\X)\) is a martingale when stopped at the boundaries of its state space.
As a consequence, the Feller--Dynkin and the martingale property are equivalent for regular diffusions on natural scale with open state space. 
By means of a counterexample, we also show that this equivalence fails for multi-dimensional diffusions.
Moreover, for It\^o diffusions we discuss relations to Cauchy problems. 
\end{abstract}

\section{The Feller--Dynkin and the Martingale Property of Diffusions}
\subsection{The Setting}
Let  \(J \subset \mathbb{R}\) be a non-empty closed, open or half open possibly infinite interval. 
We denote the interior of \(J\) by \(J^\circ,\) the closure of \(J\) in \([- \infty, \infty]\) by \(\on{cl}(J)\) and its boundary \(\on{cl}(J) \backslash J^\circ\) by \(\partial J\).
Using the classical sextuple notation of Blumenthal and Getoor, let \[\M \triangleq (\Omega, \mathcal{F}, \mathcal{F}_t \colon t \geq 0, \X_t \colon t \geq 0, \theta_t \colon t \geq 0, \P_x \colon x \in J)\] be a (path-)continuous (temporally homogeneous) conservative strong Markov process (called \emph{diffusion} in the following) with state space \((J, \mathcal{B}(J))\). 
Throughout the paper we assume that \(\M\) is \emph{regular}, i.e. \(\P_x(\t_y < \infty) > 0\) for every \(x \in J^\circ\) and \(y \in J\),
where 
\[
\t_u \triangleq \inf(t \in \mathbb{R}_+ \colon \X_t = u), \quad u \in J.
\]
As \(\M\) is a (strong) Markov process, we can define a semigroup \((T_t)_{t \geq 0}\) via
\begin{align} \label{eq: semigroup} 
	T_t f (x) \triangleq \E_x \big[ f(\X_t) \big], \quad (t, f, x) \in \mathbb{R}_+\times C_b(J) \times J.
\end{align}
It is well-known (\cite[Theorem 16.23]{breiman1968probability} or \cite[Proposition V.50.1]{RW2}) that any regular diffusion is a \emph{Feller process} in the sense that \(T_t (C_b(J)) \subset C_b(J)\). 

Next, we recall the important concepts of \emph{scale} and \emph{speed}. There exists a continuous, strictly increasing function \(\s \colon J \to \mathbb{R}\), which is unique up to increasing affine transformations, such that for any interval \(I = (a, b)\) with \([a, b] \subset J\) we have
\[
\P_x(\t_b < \t_a) = \frac{\s(x) - \s(a)}{\s(b) - \s(a)}, \quad x \in I, 
\]
see \cite[Theorem 16.27]{breiman1968probability}. A function like \(\s\) is called a \emph{scale function}. Moreover, the diffusion \(\M\) is said to be on \emph{natural scale} in case \(\on{Id}\) is a scale function. Any diffusion can be brought to natural scale via space transformation. More precisely, \(\s(\X)\) is on natural scale.
For a bounded interval \(I = (a, b)\) with \([a, b] \subset J\) we set
\begin{align} \label{eq: def G}
	G_I (x, y) \triangleq \begin{cases}
		\displaystyle\frac{2\hspace{0.03cm} (\s(x) \wedge \s(y) - \s(a))(\s(b) - \s(x) \vee \s(y))}{\s(b) - \s(a)},& x, y \in I, \\
		0,&\text{otherwise}.
	\end{cases}
\end{align}
There exists a unique Radon measure \(\m\) on \((J^\circ, \mathcal{B}(J^\circ))\) such that for any bounded interval \(I = (a, b)\) with \([a, b] \subset J\) we have 
\[
\E_x \big[ \t_a \wedge \t_b\big] = \int G_I(x, y) \m(dy), \quad x \in I,
\]
see \cite[Theorem VII.3.6]{RY}. The measure \(\m\) is called the \emph{speed measure} of \(\M\). Scale and speed determine the potential operator of the diffusion killed when exiting a bounded interval. More precisely, for any bounded interval \(I = (a, b)\) with \([a, b] \subset J\) and any Borel function \(f \colon J \to \mathbb{R}_+\) we have 
\[
\E_x \Big[ \int_0^{\t_a \wedge \t_b} f(\X_s) ds \Big] = \int G_I(x, y) f(y) \m(dy), \quad x \in I, 
\]
see \cite[Corollary VII.3.8]{RY}.
We take a reference point \(c \in J^\circ\) and define for \(x \in J^\circ\)
\begin{equation}\label{eq: u func}
	\begin{split}
		\u (x) &\triangleq \begin{cases} \displaystyle \int_{(c, x]} \m((c, z]) \s(dz),& \text{for } x \geq c,\\\vspace{-
				0.4cm}\\
			\displaystyle \int_{(x, c]} \m((z, c]) \s (dz),&\text{for } x \leq c,\end{cases}  
		\\
		\v (x) &\triangleq \begin{cases} \displaystyle \int_{(c, x]} (\s(y) - \s(c)) \m(dy),&\text{for } x \geq c,\\ \vspace{-0.4cm}
			\\
			\displaystyle\int_{(x, c]} (\s(c) - \s(y)) \m(dy),&\text{for } x \leq c.\end{cases}
	\end{split}
\end{equation}
Moreover, for \(b\in \partial J\) we write \(\u (b) \triangleq \lim_{x \to b} \u(x)\) and \(\v(b) \triangleq \lim_{x \to b} \v(x)\), where the limits are meant to be monotone. 
A boundary point \(b\in \partial J\) is called 
\begin{align*}
	\emph{regular}&\quad\text{ if } \u(b) < \infty \text{ and } \v(b) < \infty,\\  
	\emph{exit}&\quad\text{ if } \u(b) < \infty \text{ and } \v(b) = \infty,\\  
	\emph{entrance}&\quad\text{ if } \u(b) = \infty \text{ and } \v(b) < \infty,\\  
	\emph{natural}&\quad\text{ if } \u(b) = \infty \text{ and } \v(b) = \infty.
\end{align*}
These definitions are independent of the choice of the reference point \(c \in J^\circ\).
Regular and exit boundaries are called \emph{closed} or \emph{inaccessible}, and entrance and natural boundaries are called \emph{open} or \emph{accessible}. As already indicated by the names, open boundaries are not in the state space \(J\) while closed ones are, see \cite[Proposition 16.43]{breiman1968probability}.

The behavior of the diffusion at exit, entrance and natural boundaries is fully specified by \(\s\) and \(\m\).
Regular boundaries are different in this regard.
To see this, consider Brownian motion with state space \([0, \infty)\) and absorption or reflection in the origin (\cite[Section 16.3]{breiman1968probability}). In both cases the speed measure coincides with the Lebesgue measure on \((0, \infty)\) and the origin is regular.
Hence, knowing the speed measure on \(J^\circ = (0, \infty)\) does not suffice to decide whether the origin is absorbing or reflecting. There is a way to fix this issue. Namely, the speed measure \(\m\) can be extended to \((J, \mathcal{B}(J))\) such that it also encodes the behavior at regular boundary points.
In the following we explain this for \(J = [0, \infty)\) and \(\s(0) = 0\).
Define \(\s^* \colon \mathbb{R} \to \mathbb{R}\) by setting \(\s^* (x) \triangleq \s (x)\) and \(\s^*(- x) \triangleq - \s(x)\) for \(x \in \mathbb{R}_+\).
For \(I = [0, a)\) with \(a > 0\) define \(G^*_I\) as \(G_{(-a, a)}\) from \eqref{eq: def G} with \(\s\) replaced by~\(\s^*\), and set
\[
G^\circ_I (x, y) \triangleq G^*_I (x, y) + G^*_I (x, -y), \quad x, y \in \mathbb{R}_+.
\]
By \cite[Proposition VII.3.10]{RY}, it is possible to define \(\m(\{0\})\) such that for any interval \(I = [0, a)\) with \(a > 0\) and any Borel function \(f \colon \mathbb{R}_+ = J \to \mathbb{R}_+\) we have 
\begin{align} \label{eq: extended speed measure}
	\E_x \Big[ \int_0^{\t_a} f(\X_s) ds \Big] = \int_I G^\circ_I (x, y) f(y) \m(dy), \quad x \in I.
\end{align}
Let us convince ourselves that \(\m(\{0\})\) distinguishes absorption and reflection. Taking \(f \equiv \1_{\{0\}}\) in \eqref{eq: extended speed measure} yields 
\[
\E_0 \Big[ \int_0^{\t_a} \1_{\{\X_s = 0\}} ds \Big] = 2\s(a) \m(\{0\}), \quad a > 0.
\]
This formula motivates the following definitions:
A regular boundary point \(b\) is called \emph{absorbing} if \(\m(\{b\}) = \infty\), \emph{slowly reflecting} if \(0 < \m(\{b\}) < \infty\), and \emph{instantaneously reflecting} if \(\m (\{b\}) = 0\).

The scale function and the extended speed measure determine the law of the diffusion uniquely, see \cite[Corollary 16.73]{breiman1968probability}. Finally, we stress that the (extended) speed measure can also be defined via a change of time, see \cite[Theorem 33.9]{kallenberg} or \cite[Theorem~V.47.1]{RW2}. 

The above material is mainly taken from the monographs of Breiman \cite{breiman1968probability} and Revuz and Yor \cite{RY}, where the reader can find nice introductions to the topic. For a more detailed treatment we refer to the monograph of It\^o and McKean \cite{itokean74}.
At this point, we like to mention that the monographs \cite{breiman1968probability,itokean74} use different terminologies for boundary points, which are related as follows:
\begin{center}
	\begin{tabular}{ c|c } 
		Breiman&It\^o--McKean\\
		\hline
		regular&exit and entrance\\
		exit& exit not entrance\\
		entrance& entrance not exit\\
		natural& neither exit nor entrance
	\end{tabular}
\end{center}
Our terminology is taken from Breiman \cite{breiman1968probability}. It is also worth noting that the scale function is defined consistently in the literature, while different normalizations appear for the speed measure. For example, the speed measure in It\^o and McKean \cite{itokean74} and Revuz and Yor \cite{RY} is twice the speed measure from Breiman \cite{breiman1968probability} and Rogers and Williams \cite{RW2}. We use the scaling from \cite{breiman1968probability,RW2}.

\subsection{Equivalence of the Feller--Dynkin and the Martingale Property}
Let \(C_0(J)\) be the Banach space of continuous functions \(J \to \mathbb{R}\) which are vanishing at infinity endowed with the sup-norm. The process \(\X\) is called a \emph{Feller--Dynkin (FD) process} if the semigroup \((T_t)_{t \geq 0}\), which we defined in \eqref{eq: semigroup}, is a strongly continuous semigroup on~\(C_0(J)\). 
We define the stopping time
\[
\i \triangleq \inf (t \in \mathbb{R}_+ \colon \X_t \not \in J^\circ).
\]
It is well-known (\cite[Corollary V.46.15]{RW2}) that the stopped process \(\Y \triangleq \s\hspace{0.02cm}(\X_{\cdot \wedge \i})\) is a local \(\P_x\)-martingale for all \(x \in J^\circ\).	The following theorem is our main result.
\begin{theorem}\label{theo: main1}
	The following are equivalent:
	\begin{enumerate}
		\item[\textup{(i)}] \(\X\) is an FD process.
		\item[\textup{(ii)}] \(\Y\) is a \(\P_x\)-martingale for every \(x \in J^\circ\).
		\item[\textup{(iii)}] Every open boundary point is natural.
	\end{enumerate}
\end{theorem}
As an immediate consequence of Theorem \ref{theo: main1}, we obtain the following:
\begin{corollary}
	Suppose that \(\X\) is on natural scale and that all regular boundaries are absorbing. Then, \(\X\) is an FD process if and only if it is a \(\P_x\)-martingale for all \(x \in J^\circ\).
\end{corollary}
\begin{proof}
	We recall that exit boundaries are always absorbing in the sense that they cannot be left by the diffusion, see \cite[Problem 14, p. 370]{breiman1968probability}. Hence, under the assumption that all regular boundaries are absorbing, we have a.s. \(\X = \X_{\cdot \wedge \i}\). Thanks to this observation, the claim follows directly from Theorem \ref{theo: main1}.
\end{proof}
\begin{remark}
	Urusov and Zervos \cite{urusov17} proved that (iii) in Theorem \ref{theo: main1} is equivalent to the martingale properties of the so-called \emph{\(r\)-excessive local martingales}. By virtue of Theorem \ref{theo: main1}, their result provides another characterization of the FD property in terms of martingale properties.
\end{remark}
\begin{remark}
	A standard example for a non-FD process is the three-dimensional Bessel process (denoted \(\on{Bes}^3\); \cite[Section VI.3]{RY}) and its inverse is a standard example for a strict local martingale. 
	These examples are connected via Theorem \ref{theo: main1} as \(\s (x) = -1/x\) for~\(x > 0\) is a scale function of \(\on{Bes}^3\). Let us stress that the state space of \(\on{Bes}^3\) is necessarily \((0, \infty)\) as it is otherwise no \emph{regular} diffusion.
\end{remark}

Our contribution in Theorem \ref{theo: main1} is the equivalence of (i) and (ii), which we think is quite surprising. In Section \ref{sec: literature} below we comment in detail on related literature.
In the proof of Theorem \ref{theo: main1}, which is given in Section \ref{sec: pf} below, we will see that \(\Y\) is a true martingale if \(\X\) needs a long time to get close to open boundary points and that \(\X\) is an FD process if it needs a long time to get away from them. It seems to be a coincidence that these properties are equivalent. Indeed, as we discuss in Section \ref{sec: multi}, the equivalence of the FD and the martingale property is a one-dimensional phenomenon.

It is well-known (\cite[Theorem 33.9]{kallenberg} or \cite[Theorem V.47.1]{RW2}) that any regular diffusion on natural scale is a time change of Brownian motion. 
On page 280 of their monograph \cite{RW1} Rogers and Williams write the following: \emph{Deciding whether or not the FD property is preserved under probabilistic operations such as time-substitution is generally a very difficult problem.} In the same spirit, it is well-known that the \emph{local} martingale property is preserved by time changes (\cite[Section V.1]{RY}), but not necessarily the \emph{true} martingale property. 
By virtue of these observations, the equivalence of (i) and (ii) in Theorem \ref{theo: main1} can be seen as follows: The time change related to the diffusion \(\X\) preserves the martingale property of the underlying Brownian motion precisely when it preserves its FD property.

\subsection{Comments on Related Literature} \label{sec: literature}
The question when a non-negative It\^o diffusion with dynamics
\begin{align*} 
	d \X_t = \sigma (\X_t) d W_t, \quad W = \text{Brownian motion},
\end{align*}
is a true martingale is e.g. interesting for mathematical finance, where the martingale property decides about the absence and existence of certain arbitrage opportunities. Motivated by such an application, Delbaen and Shirakawa \cite{Delbaen2002} proved an analytic integral test for the martingale property. 
Later, Kotani \cite{Kotani2006} and Hulley \cite{Hulley2009} gave answers for general regular diffusions on natural scale via integral tests depending on the speed measure. More precisely, the equivalence of (ii) and (iii) in Theorem \ref{theo: main1} is their result.

The quite different question whether an It\^o diffusion with drift is an FD process was studied by Feller \cite{Feller} and Cl\'ement and Timmermans \cite{clem86} from an analytic perspective, and by Azencott \cite{Azencott1974} from a more probabilistic point of view. We emphasis that Azencott was also interested in higher dimensional settings. These references provide the equivalence of (i) and (iii) in Theorem~\ref{theo: main1} for certain It\^o diffusions. Again for It\^o diffusions, the implication (iii) \(\Rightarrow\) (i) is also given in the monograph of Ethier and Kurtz, see \cite[Corollary~8.1.2]{EK}. Kallenberg (\cite[Theorem~33.13]{kallenberg}) proved the following related result: Form \(\overline{J}\) via attaching entrance boundaries of \(\X\) to \(J\). Then, \(\X\) extends to an FD process on \(\overline{J}\). As \(\overline{J} = J\) in case all open boundary points are natural, this theorem also implies the implication (iii) \(\Rightarrow\) (i) from Theorem \ref{theo: main1}.

Among other things, Eberle \cite{eberle} studied whether for a regular second order differential operator \(\mathcal{L}\) on \(C_c^\infty(J^\circ)\) there exists at most one strongly continuous semigroup on a suitable weighted \(L^p\) space whose generator extends \(\mathcal{L}\). If this is the case, \(\mathcal{L}\) is said to be \emph{\(L^p\)-unique}.
As noted in \cite[Remark, p. 3]{eberle}, roughly speaking uniqueness of FD semigroups can be viewed as some limit of \(L^p\)-uniqueness as \(p \to \infty\). By this intuition, the equivalence of (i) and (iii) from Theorem \ref{theo: main1} mirrors the results summarized in \cite[(i) -- (iv), pp. 174 -- 176]{eberle}: Provided no boundary point is regular, \(L^p\)-uniqueness holds for all large \(p > 1\) if and only if open boundaries are natural. In case the boundary contains a regular point it does not suffice to specify the speed measure on \(J^\circ\) and any reasonable type of uniqueness fails, see also \cite[(i), p. 174]{eberle} for a similar comment. 

Our main observation is the equivalence of (i) and (ii) in Theorem \ref{theo: main1}. The purpose of this paper is to report this phenomenon and, as we find it not intuitive, to explain it via a complete and (mainly) self-contained proof, which borrows and connects many ideas from~\cite{Azencott1974,Hulley2009,Kotani2006}.

\subsection{A Counterexample for the Multi-Dimensional Case} \label{sec: multi}
It is natural to ask whether the equivalence of (i) and (ii) from Theorem \ref{theo: main1} also holds in a multi-dimensional setting. In this section we give an example, inspired by a comment on page 238 in \cite{Azencott1974}, which shows that this is not the case. In other words, the equivalence of the FD and the martingale property is a one-dimensional phenomenon.

Take \(d \geq 2\), define \(\Omega \triangleq C(\mathbb{R}_+, \mathbb{R}^d)\) and denote the coordinate process by \(\B =(\B_t)_{t \geq 0}\). Let \(\mathcal{F}\) and \((\mathcal{F}_t)_{t \geq 0}\) be the \(\sigma\)-field and the (right-continuous) filtration generated by~\(\B\). Furthermore, let \(\Q_x\) be 
the \(d\)-dimensional Wiener measure with starting point \(x \in \mathbb{R}^d\).
Let \(D \subset \mathbb{R}^d\) be a non-empty domain of finite Lebesgue measure.
A point \(\I \in \partial D\) is called \emph{irregular} if
\(
\Q_{\I} (\S= 0)= 0\) with \(\S\triangleq \inf(t > 0 \colon \B_t \not \in D).\) 
Irregularity can also be defined via the Dirichlet problem, see \cite[Theorem 4.2.2]{portstone}.
The set of irregular points is denoted by \(\mathcal{I}\). Note that \(\Q_\I (\S = 0) = 1\) for all \(\I \in \partial D \backslash \mathcal{I}\) by Blumenthal's zero-one law.

\begin{example} \label{ex: irregular}
	\begin{enumerate}
		\item[\textup{(i)}] If \(D \equiv \{x \in \mathbb{R}^d \colon 0 < \|x\| < 1\}\), then \(\mathcal{I} = \{0\}\).
		\item[(\textup{ii)}] An example for a domain with a connected boundary containing an irregular point is \emph{Lebesgue's thorn}, see \cite[Example 4.2.17]{KaraShre}. 
	\end{enumerate}
\end{example}
Define \[D' \triangleq \on{cl}(D) \backslash \mathcal{I},\qquad
\P_x \triangleq \Q_x \circ \B_{\cdot \wedge \T}^{-1} \colon \ x \in D'.\]
Finally, we set
\[
\M \triangleq (\Omega, \mathcal{F}, \mathcal{F}_t \colon t \geq 0, \B_t \colon t \geq 0, \theta_t\colon t \geq 0, \P_x \colon x \in D'),
\]
where \((\theta_t)_{t \geq 0}\) is the usual shift operator on \(\Omega\), i.e. \(\theta_s \omega (t)= \omega(t + s)\) for \(\omega \in \Omega\) and~\(s, t \in \mathbb{R}_+\).
\begin{theorem} \label{prop: CE}
	\(\M\) is a strong Markov process with state space \((D', \mathcal{B}(D'))\) and \(\B\) is a \(\P_x\)-martingale for every \(x \in D\). Moreover, \(\M\) is an FD process if and only if \hspace{0.05cm}\(\mathcal{I} = \emptyset\). 
\end{theorem}	
\begin{discussion}
	To see the connection of Theorems \ref{theo: main1} and \ref{prop: CE}, note that an irregular boundary point can be viewed as a multi-dimensional version of an entrance boundary: A point \(\I \in \partial D\) is irregular if Brownian motion started in \(\I\) enters \(D\) immediately. Further, as \(\mathcal{I}\) is a polar set (\cite[Theorem 2.6.3]{portstone}), Brownian motion never hits \(\mathcal{I}\) when started in~\(D'\). 
	Hence, roughly speaking, Theorem \ref{prop: CE} shows that \(\M\) is an FD process if and only if there are no entrance boundary points, which is also the equivalence of (i) and (iii) in Theorem~\ref{theo: main1}. We point to the difference that entrance boundaries are necessarily infinite for diffusions on natural scale, while irregular points are elements of \(\mathbb{R}^d\). This is related to the well-known fact (\cite[Proposition~2.3.2]{portstone}) that there are \emph{no} irregular points in the one-dimensional case. Indeed, for \(d = 1\) the system \(\M\) is also known to be an FD process (this is confirmed by Theorem \ref{theo: main1}).
	In contrast to the FD property, irregular points do not affect the martingale property.
	
	We have excluded the set \(\mathcal{I}\) of irregular points from the state space as this seems to be closer to the one-dimensional setting in which entrance boundaries are excluded by regularity of the diffusion.
\end{discussion}

\begin{proof}[Proof of Theorem \ref{prop: CE}]
	The strong Markov property of \(\M\) can be proved as in \cite[Section~3.9, pp. 102 -- 103]{itokean74}. 
	
	The martingale property follows from those of Brownian motion and the optional stopping theorem. 
	To see this, first note that \(\X_{\cdot \wedge \T}^{-1} (\mathcal{F}_t) \subset \mathcal{F}_{t \wedge \T}\) for all \(t \in \mathbb{R}_+\). Then, the optional stopping theorem yields that for all \(s<t\) and \(G \in \mathcal{F}_s\) we have \(\B_t, \B_s \in L^1 (\P_x)\) and
	\[
	\E^{\P_x} \big[ \B_t \1_G \big] = \E^{\Q_x} \big[ \B_{t \wedge \T} \1_{\{\B_{\cdot \wedge \T} \in G\}} \big] = \E^{\Q_x} \big[ \B_{s \wedge \T} \1_{\{\B_{\cdot \wedge \T} \in G\}} \big] = \E^{\P_x} \big[ \B_s \1_G \big].
	\]
	This is the martingale property. 
	
	If \(\mathcal{I} = \emptyset\), then \(\M\) is an FD process by \cite[Theorem 4.1.9]{knight}.
	We now show the converse direction, i.e. we assume that \(\mathcal{I} \not = \emptyset\) and we 
	take \(\I\in \mathcal{I}\). Thanks to \cite[Proposition~4.2.14]{portstone}, there exists a compact set \(K \subset D'\) such that
	\begin{align}\label{eq: PS}
		\limsup_{\substack{x \to \I\\ x \in D}} \Q_{x} (T_K < \T) > 0, \quad T_K \triangleq \inf (t > 0 \colon \B_t \in K).
	\end{align}
	Furthermore, by \cite[(X), p. 148]{chung} and the assumption that \(D\) has finite Lebesgue measure, there exists an \(\alpha > 0\) such that 
	\begin{align}\label{eq: exp mom hit}
		\sup_{x \in \mathbb{R}^d} \E^{\Q_x} \big[ e^{\alpha \T} \big] < \infty.
	\end{align}
	Using Galmarino's test (\cite[p.~86]{itokean74}) and the Cauchy--Schwarz inequality, for all \(x \in D\) we obtain 
	\begin{equation}\label{eq: CS ineq} \begin{split}
			\Q_x (T_K < \T) &= \Q_x (T_K(\B_{\cdot \wedge \T}) < \T) \\[3pt]&\leq \E^{\Q_x} \big[ e^{\alpha (\T- T_K(\B_{\cdot \wedge \T}))/2} \big]\vspace{0.05cm} \\&\leq \sup_{z \in \mathbb{R}^d}\E^{\Q_z} \big[ e^{\alpha \T} \big]^{\frac{1}{2}} \E^{\P_x} \big[ e^{- \alpha T_K} \big]^{\frac{1}{2}}.
		\end{split}
	\end{equation}
	By \cite[Remark 1]{criens2020}, we have
	\begin{align} \label{eq: rem SPA}
		\M \text{ is an FD process } \Rightarrow \ \lim_{\substack{x \to \I\\ x \in D}} \E^{\P_{x}} \big[ e^{- \alpha T_K} \big] = 0.
	\end{align}
	Finally, \eqref{eq: PS} -- \eqref{eq: rem SPA} yield that \(\M\) is no FD process.
\end{proof}

\subsection{Equivalence of Cauchy Problems in It\^o Diffusion Settings}
It is well-known that the FD and the martingale property have close relations to existence and uniqueness properties of Cauchy problems. Thanks to Theorem \ref{theo: main1}, we can connect these relations.

Suppose that \(J = (l, r)\) for \(- \infty \leq l < r \leq \infty\) and that
\[
\s (x) \triangleq \int_{c}^x \exp \Big( - \int_c^\xi \frac{2 b(z) dz}{\sigma^2(z)} \Big) d\xi, \qquad \m(dx) \triangleq \frac{dx}{\s' (x) \sigma^2(x)},
\]
where \(c \in J\) is an arbitrary reference point and \(b \colon J \to \mathbb{R}\) and \(\sigma \colon J \to \mathbb{R} \backslash \{0\}\) are continuous functions.
Moreover, we set
\[
S f \triangleq b f' + \tfrac{\sigma^2}{2} f''\ \text{ for }\ f \in D (S) \triangleq \big\{f \in C_0 (J) \cap C^2 (J)\colon Sf \in C_0 (J) \big\}.
\]
\begin{remark}
	In case \(\X\) is an FD process, it is known that \((S, D(S))\) is its infinitesimal generator, see \cite[Corollary 8.1.2]{EK}.
\end{remark}
We start with a consequence of a main result from \cite{bay2010} which relates (ii) and (iii) from Theorem \ref{theo: main1} to existence and uniqueness of a classical solution to a certain Cauchy problem with boundary datum of linear growth.
\begin{theorem} \label{theo: bay repeat}
	Suppose that \(J = (0, \infty), b \equiv 0\) and that \(\sigma\) is locally H\"older continuous with exponent \(1/2\). Then, \textup{(i) -- (iii)} from Theorem \ref{theo: main1} are equivalent to the following:
	\begin{enumerate}
		\item[\textup{(iv)}] For every continuous function \(g \colon \mathbb{R}_+ \to \mathbb{R}_+\) of linear growth, i.e. \(|g(x)| \leq C(1 + |x|)\) with~\(C > 0\), and any finite time horizon \(T > 0\) the Cauchy problem 
		\begin{align*}
			\begin{cases}
				\frac{du}{dt} + \frac{\sigma^2}{2} u'' = 0,& \text{on } (0, \infty) \times [0, T),\\
				u(0, t) = g(0),& t \in [0, T],\\
				u(x, T) = g(x),& x \in \mathbb{R}_+, 
			\end{cases}
		\end{align*}
		has a unique solution \(u \colon \mathbb{R}_+ \times [0, T] \to \mathbb{R}\) such that \(u \in C^{2, 1} ((0, \infty) \times [0, T))\).
	\end{enumerate}
\end{theorem}
\begin{proof}
	The equivalence of (iii) and (iv) follows from \cite[Theorem 2]{bay2010}.
\end{proof}
In case (iii) fails, it has been shown in \cite{cet18, cet20} that for appropriate boundary data the associated Cauchy problem still has a solution which is unique among all solutions with certain non-standard boundary behavior. 

The proof of Theorem \ref{theo: bay repeat} in \cite{bay2010} uses PDE theory in combination with uniform integrability properties, which stem from the martingale property of \(\X\), i.e. item (ii) from Theorem \ref{theo: main1}.

Next, we provide another characterization of (i) -- (iii) from Theorem \ref{theo: main1} in terms of properties of Cauchy problems. 

\begin{theorem}\label{prop: Cauchy}
	\textup{(i) -- (iii)} from Theorem \ref{theo: main1} are equivalent to each of the following:
	\begin{enumerate}
		\item[\textup{(v)}]
		For all \(g \in D(S)\) the Cauchy problem 
		\[
		\frac{du}{dt} = S u, \quad u (0) = g,
		\]
		has a unique solution \(u \colon \mathbb{R}_+ \to C_0(J)\) which is a continuously differentiable function such that \(u (t) \in D(S)\) for all \(t > 0\).
		\item[\textup{(vi)}] For all \(g \in D(S)\) there exists a continuous function \(u \colon \mathbb{R}_+ \to C_0(J)\) such that \(u (0) = g, u (t) \in D(S)\) for all \(t > 0\), \(S u \colon (0, \infty) \to C_0(J)\) is continuous, and 
		\[
		u (t) - u(\varepsilon) = \int_{\varepsilon}^t S u (s) ds
		\]
		for all \(t > \varepsilon > 0\).
	\end{enumerate}
\end{theorem}
\begin{proof}
	If (iii) holds, \cite[Corollary 8.1.2]{EK} and \cite[Theorem 4.1.3]{pazy} yield (v). Obviously, (v) implies (vi). Suppose that (vi) holds. As \(C_c^2(J)\) is dense in \(C_0(J)\) and \(C_c^2(J) \subset D(S)\), the operator \((S, D(S))\) is densely defined. Furthermore, it follows from \cite[Proposition~1]{clem86} that \((S, D(S))\) is closed and dissipative. Hence, (vi) and \cite[Proposition 1.3.4]{EK} yield that \((S, D(S))\) is the generator of a strongly continuous semigroup on \(C_0(J)\). Now, it follows verbatim as in the proof of \cite[Lemma 3]{clem86} that there exist two positive monotone solutions \(u_l\) and \(u_r\) to \(u = S u\) such that \(\lim_{x \to l} u_l (x) = \lim_{x \to r} u_r (x) = 0\). As in the proof of Lemma~\ref{lem: func} below, if \(l\) is not natural then there exists a positive increasing solution \(u^*_l\) to \(u = Su\) with \(\lim_{x \to l} u^*_l (x) > 0\). However, since \(u^*_l = c u_l\) for \(c > 0\) (see \cite[p. 129]{itokean74}), this yields a contradiction. The same argument shows that \(r\) is natural. The proof is complete.
\end{proof}

It is interesting to observe that for the Cauchy problem from (iv) uniqueness fails in case (i) -- (iii) from Theorem \ref{theo: main1} fail, see the proof of \cite[Theorem 2]{bay2010}. In other words, existence is not the decisive property in Theorem~\ref{theo: bay repeat}. This is quite different for the Cauchy problem from~(v). In case it has a solution (for all initial data), then (vi) holds and (i) -- (iii) from Theorem~\ref{theo: main1} hold, too.

\section{Proof of Theorem \ref{theo: main1}} \label{sec: pf}
As the scale function \(\s\) is continuous and strictly increasing, \(\s \colon J \to \s(J)\) is a homeomorphism and, by virtue of \cite[Exercise VII.3.18]{RY}, \(\s (\X)\) is a regular diffusion with state space \(J^* \triangleq \s(J)\), scale function \(\on{Id}\) and speed measure \(\m \circ \s^{-1}\).
Note the following implications: If \(f \in C_0 (J)\) then \(f \circ \s^{-1} \in C_0 (J^*)\), and if \(f \in C_0 (J^*)\) then \(f \circ \s \in C_0 (J)\). Thus, \(\X\) and \(\s(\X)\) are simultaneously  FD processes.
With this observation in mind, we can and will w.l.o.g. assume that \(\X\) is on natural scale, i.e. \(\s = \on{Id}\).

\begin{lemma}\label{lem: main1}
	\begin{enumerate}
		\item[\textup{(i)}]
		\(\X\) is an FD process if and only if 
		\begin{align}\label{eq: equiv fd}
			\lim_{x \to b} \E_x \big[ e^{- \alpha \t_y} \big] = 0 \text{ for all } y \in J^\circ, \alpha > 0 \text{ and any infinite \(b \in \partial J\)}.\end{align}
		\item[\textup{(ii)}]
		\(\X_{\cdot \wedge \i}\) is a \(\P_x\)-martingale for all \(x \in J^\circ\) if and only if 
		\begin{align}\label{eq: equiv mg}
			\lim_{y \to b} y \E_x \big[ e^{- \alpha \t_y} \big] = 0 \text{ for all } x \in J^\circ, \alpha > 0 \text{ and any infinite \(b \in \partial J\)}.\end{align}
	\end{enumerate}
\end{lemma} 
Part (i) of Lemma \ref{lem: main1} is a version of \cite[Proposition 3.1]{Azencott1974} and \cite[Remark 1]{criens2020} for our framework. 
In \cite{Azencott1974} the result is shown for multi-dimensional It\^o diffusions with (locally) H\"older coefficients and in \cite{criens2020} it is shown in a general martingale problem framework.
The general idea for its proof given below is borrowed from \cite{Azencott1974}.
The argument in \cite{Azencott1974} for the \emph{only if} implication uses analytic tools. The proof given below borrows the supermartingale argument from \cite{criens2020}. 
Part (ii) can be extracted from \cite{Hulley2009}, although it has not been stated there in this form. Below we give a proof for which we borrow arguments from the proof of \cite[Theorem~3.9]{Hulley2009}. 

As every regular diffusion is already a Feller process, it is an FD process if and only if \(T_t f\) vanishes at infinity for all \(f \in C_0(J)\) and \(t > 0\).
Thus, \(\X\) should be an FD process precisely if it stays some time close to open boundaries. Part (i) of Lemma~\ref{lem: main1} quantifies this intuition. At this point we stress that regular diffusions on natural scale always stay some time close to finite open boundaries. This explains why only the infinite boundaries are mentioned in \eqref{eq: equiv fd}. As we have seen in Section \ref{sec: multi} this is quite different in the multi-dimensional setting.

To get an idea for part (ii), consider \(J = (0, \infty)\) and note that for all~\(y \geq x\) the stopped process \(\X_{\cdot \wedge \t_y}\) is a bounded local \(\P_x\)-martingale and consequently, a \(\P_x\)-martingale. The condition \eqref{eq: equiv mg} can be viewed as a criterion for the uniform \(\P_x\)-integrability of \(\{\X_{t \wedge \t_y} \colon y \geq x\}\) for every \(t > 0\), which is necessary and sufficient for the \(\P_x\)-martingale property of \(\X\). To get an intuition for this, recall the criterion of de la Vall\'ee Poussin: A family \(\Pi\subset L^1\) is uniformly integrable if and only if there exists a convex monotone function \(H \colon \mathbb{R}_+ \to \mathbb{R}_+\) such that \(\sup_{X \in \Pi} \E[H(|X|)] < \infty\) and \(H(x)/x \to \infty\) as \(x \to \infty\). The condition \eqref{eq: equiv mg} mirrors this criterion with \(H (x) = 1/\E_y [ e^{- \alpha \t_x}]\) for \(x > y\).

At first glance \eqref{eq: equiv mg} seems to be stronger than \eqref{eq: equiv fd}. For example, suppose that \(\g (x, y) \triangleq \E_x[e^{- \alpha \t_y}]\) is symmetric in \(x, y \in J^\circ\). Then, \eqref{eq: equiv mg} clearly implies \eqref{eq: equiv fd}. 
It turns out that this situation is quite special: \(\g\) is symmetric if and only if the diffusion \(\X\) behaves like a Brownian motion up to a constant scale factor in the interior of its state space \(J\).\footnote{
	Let \(g_1\) and \(g_2\) be as in the proof of Lemma \ref{lem: main2}. Symmetry of \(\g\) means \(g_1 g_2 = 1\). Let \(B\) be the Wronskian, i.e. \(\text{const.} \equiv B = g^+_1 g_2 - g_1 g^+_2\), see \cite[p. 130]{itokean74}. Using the product rule we obtain
	\[
	0 = (g_1 g_2)^+ = g_1^+ g_2 + g_1 g_2^+ = \begin{cases}2 g_1^+ g_2 - B = 2 g_1^+/g_1 - B,\\
	2 g_1 g^+_2 + B = 2 g^+_2 /g_2 + B,
	\end{cases}
	\]
	which means
	 \(g_1^+ = Bg_1 / 2\) and \(g^+_2 = - B g_2 /2\). 
	 Using these identities, \(d g^+_i = 2 \alpha g_i d \m, i = 1,2,\) and integration by parts, we obtain for all \(a, b \in J^\circ\) with \(a < b\) that
	\begin{align*}
	0 &= \int_{(a, b]} d (g_1 g_2)^+ 
	= \int_{(a, b]} g_1^+ d g_2 + \int_{(a, b]} g_2 d g^+_1 + \int_{(a, b]} g_2^+ d g_1 + \int_{(a, b]} g_1 d g^+_2
	\\&= 2 \int_a^b g^+_1 g_2^+ dx + 4 \alpha \int_{(a, b]} g_1 g_2 d \m = - \frac{B^2}{2} (a - b) + 4 \alpha \m((a, b]).
	\end{align*}
	Consequently, 
	\(\m (dx) = \textup{const. }dx\) on \((J^\circ, \mathcal{B}(J^\circ))\).
	Hence, \(\X\) behaves like a Brownian motion up to a constant scale factor in the interior of \(J\).} 
In case \(\M\) is a Brownian motion, it is easy to show that 
\(
\E_x[e^{- \alpha \t_y}] = e^{- \sqrt{2 \alpha}|x - y|}\) for all \(x, y \in \mathbb{R}
\)
and both \eqref{eq: equiv fd} and \eqref{eq: equiv mg} are satisfied. We have the following general relation:

\begin{lemma}\label{lem: main2}
	\eqref{eq: equiv fd} \(\Leftrightarrow\) \eqref{eq: equiv mg} \(\Leftrightarrow\) all infinite boundaries are natural.
\end{lemma} 

Lemma \ref{lem: main2} shows that \(\X\) approaches infinite boundaries slow enough to be a martingale precisely when it needs long enough to get away from them to be an FD process. This connection seems to be a surprising coincidence. Lemma \ref{lem: main2} is known in different formulations, see \cite[Propositions 3.12 and 3.13]{Hulley2009}, \cite[Table 1, p. 130]{itokean74}, \cite[Lemma~3]{Kotani2006} or \cite[Theorem 2.2]{urusov17}.  Below we give a complete proof, which borrows ideas from these references. We think that its analytic character supports the impression that the equivalence of the FD and the martingale property is quite surprising.

\begin{proof}[Proof of Theorem \ref{theo: main1}]
	Lemmata \ref{lem: main1} and \ref{lem: main2} imply Theorem \ref{theo: main1}.
\end{proof}

\begin{proof}[Proof of Lemma \ref{lem: main2}]
	Fix \(\alpha > 0\) and a reference point \(y \in J^\circ\). Using the notation of It\^o and McKean (\cite[pp. 128]{itokean74}), for \(x \in J^\circ\) we define 
	\begin{align*}
		g_1 (x) &\triangleq \begin{cases} \E_x \big[ e^{- \alpha \t_y} \big], & x \leq y,\\
			1/\E_y \big[ e^{- \alpha \t_x}\big],& y < x,\end{cases}
		\\
		g_2 (x) &\triangleq \begin{cases}1/\E_y \big[ e^{- \alpha \t_x}\big],& x \leq y,
			\\\E_x \big[ e^{- \alpha \t_y} \big], & y < x.\end{cases}
	\end{align*}
	It is well-known (\cite[Proposition V.50.3]{RW2}) that \(g_1\) and \(g_2\) are strictly convex, continuous, strictly monotone, and positive and finite (throughout \(J^\circ\)). More precisely, \(g_1\) is strictly increasing and \(g_2\) is strictly decreasing. Furthermore, \(g_1\) and \(g_2\) both solve the differential equation
	\[
	\frac{1}{2 \alpha} \frac{d}{d\m} \frac{d^+ g}{dx}  = g,
	\]
	that is for \(z, y \in J^\circ\) with \(z < y\)
	\[
	\frac{d^+ g}{dx} (y) - \frac{d^+ g}{dx} (z) = 2 \alpha \int_{(z, y]} g (u) \m (du).
	\]
	
	\noindent
	\emph{Case 1: \(\infty\) is a boundary point of \(J\).}
	Clearly, for \(b = \infty\) the property \eqref{eq: equiv fd} means that \(g_2 (\infty) \triangleq \lim_{x \to \infty} g_2 (x) = 0\), and  \eqref{eq: equiv mg} means that \begin{align}\label{eq: eq1}\lim_{x \to \infty} \frac{x}{g_1 (x)} = 0.\end{align}
	We now translate \eqref{eq: eq1} to a property of \(g^+_1 \triangleq d^+ g_1/dx\).
	As \(g_1\) is convex, we have 
	\begin{align*}
		\frac{g_1 (x) - g_1(z)}{x - z} \leq g^+_1 (x), \quad x, z \in J^\circ, x > z,
	\end{align*}
	which shows that
	\eqref{eq: eq1} implies \(g^+_1 (\infty) \triangleq \lim_{x \to \infty} g^+_1 (x) = \infty.
	\)
	Conversely, L'Hopital's rule (see \cite[Theorem 3]{Vyborny89} for a suitable version with right derivatives) yields that \eqref{eq: eq1} is implied by
	\(
	g^+_1 (\infty) = \infty\).
	Thus,  \eqref{eq: eq1} is equivalent to \(g_1^+ (\infty) = \infty\). We claim the following:
	\begin{align}\label{eq: imp}
		g^+_1 (\infty) = \infty \ \Rightarrow \ g_2 (\infty) = 0 \ \Rightarrow \ \infty\text{ is natural}\ \Rightarrow \ g^+_1 (\infty) = \infty.
	\end{align}
	These implications yield the equivalences in Lemma \ref{lem: main2} for the boundary point \(\infty\).
	\\
	
	\noindent
	\emph{Proof of 1st implication in \eqref{eq: imp}:} By \cite[Theorem V.50.7]{RW2} (or \cite[p. 130]{itokean74}), the Wronskian is constant, i.e.
	\(
	g_2 g^+_1 - g_1 g^+_2 \equiv \text{constant} \triangleq B.
	\)
	Now, \(g_2g_1^+ \leq B\) shows that \(g_1^+ (\infty) = \infty \Rightarrow g_2 (\infty) = 0\).
	\\
	
	\noindent
	\emph{Proof of 2nd implication in \eqref{eq: imp}:} 
	\begin{lemma} \label{lem: func}
		If \(\infty\) is not natural, then there exists a continuous and decreasing function \(g \colon J^\circ \to [1, \infty)\) such that \(	\frac{1}{2 \alpha} \frac{d}{d\m} \frac{d^+g}{dx}  =  g\) and \(\lim_{x \to \infty} g(x) \triangleq g (\infty) = 1.\)
	\end{lemma}
	\begin{proof}
		We mimic the proof of \cite[Lemma 5.5.26]{KaraShre} (see also \cite[Section II.2]{mandl}). 		Assume that \(\infty\) is not natural. 
		Set \(u_0 = 1\) and 
		\[
		u_n (x) \triangleq \int_x^\infty \int_{(y, \infty)} u_{n - 1}(z) \m (dz) dy = \int_{(x, \infty)} (z - x) u_{n - 1} (z) \m (dz),
		\]
		for \(x \in J^\circ\) and \(n = 1, 2, \dots.\)
		We stress that \(u_1, u_2, \dots\) are well-defined, continuous and decreasing, because \(\infty\) is \emph{not} natural. 
		Induction shows that
		\begin{align}\label{eq: first bound}
			u_n \leq \frac{u_1^n}{n!}, \quad n = 1, 2, \dots.
		\end{align}
		Indeed, the case \(n = 1\) is clear and if the inequality holds for \(n \in \mathbb{N}\), then
		\begin{align*}
			u_{n + 1} &= \int_\cdot^\infty \int_{(y, \infty)} u_n (z) \m(dz) dy 
			\leq \frac{1}{n!}\int_\cdot^\infty u_1^n(y) \m ((y, \infty)) dy 
			\\&= \frac{- 1}{n!} \int_\cdot^\infty u_1^n(y) u_1 (dy)
			= \frac{u^{n + 1}_1}{(n + 1)!}.
		\end{align*}
		Using \eqref{eq: first bound}, we also get
		\begin{align} \label{eq: second bound}
			\Big|\frac{d^+ u_n}{dx}\Big| \leq \frac{u_1^{n - 1}}{(n - 1)!} \m ((\hspace{0.05cm}\cdot\hspace{0.05cm}, \infty)), \quad n = 1, 2, \dots.
		\end{align}
		Thanks again to \eqref{eq: first bound}, 
		\(
		g \triangleq \sum_{n = 0}^\infty (2\alpha)^n u_n\)
		defines a continuous and decreasing function.
		We also see that
		\(
		1 + 2 \alpha u_1 \leq g \leq e^{2 \alpha u_1}
		\)
		and consequently, \(g (\infty) = 1\). Moreover, using~\eqref{eq: second bound}, we get
		\begin{align*}
			\frac{d^+ g}{dx} &= \sum_{n = 1}^\infty (2\alpha)^n \frac{d^+ u_n}{dx} = \sum_{n = 1}^\infty (2 \alpha)^n (- 1) \int_{(\hspace{0.02cm}\cdot\hspace{0.02cm}, \infty)} u_{n - 1} (z) \m (dz)
			\\&= - 2\alpha \int_{(\hspace{0.02cm}\cdot\hspace{0.02cm}, \infty)} \sum_{n = 0}^\infty (2 \alpha)^n u_{n} (z) \m (dz) 
			= - 2 \alpha \int_{(\hspace{0.02cm}\cdot\hspace{0.02cm}, \infty)} g (z) \m (dz).
		\end{align*}
		For \(y, z \in J^\circ\) with \(y < z\) this shows that
		\[
		\frac{d^+ g}{dx} (z) - \frac{d^+ g}{dx} (y) = 2 \alpha \int_{(y, z]} g (x) \m (dx),
		\]
		which is nothing else than
		\(
		\frac{1}{2 \alpha} \frac{d}{d \m} \frac{d^+ g}{dx} = g.
		\)
		In summary, \(g\) has all claimed properties.
	\end{proof}
	Assume that \(\infty\) is not natural and take \(g\) as in Lemma \ref{lem: func}. Then, the uniqueness theorem \cite[Theorem~16.69]{breiman1968probability} implies that \(g = c\hspace{0.05cm} g_2\) for a constant \(c > 0\). Thus, \(g_2 (\infty) > 0\) and we conclude that \(g_2(\infty) = 0\) \(\Rightarrow\) \(\infty\) is natural.
	\\
	
	\noindent
	\emph{Proof of 3rd implication in \eqref{eq: imp}:}
	Assume that \(g^+_1 (\infty) < \infty\). Then, using the subdifferential inequality, we obtain for every \(a \in J^\circ\) that
	\begin{align*}
		\int_{(a, \infty)} (z - a) \m(dz) &\leq \int_{(a, \infty)} \frac{g_1 (z) \m(dz)}{g^+_1(a)} 
		= \frac{g_1^+ (\infty) - g_1^+ (a)}{2 \alpha g^+_1 (a)} < \infty.
	\end{align*}
	Consequently, \(\infty\) cannot be natural. We conclude that \(\infty\) is natural \(\Rightarrow\) \(g^+_1 (\infty)= \infty\).
	\\
	
	\noindent
	\emph{Case 2: \(- \infty\) is a boundary point of \(J\).} In this case \eqref{eq: equiv fd} means that \(g_1 (- \infty) \triangleq \lim_{x \to - \infty} g_1 (x) = 0\), and \eqref{eq: equiv mg} means that \(\lim_{x \to - \infty} \frac{x}{g_2(x)} = 0\). As in the previous case, we see that 
	\[
	\lim_{x \to - \infty} \frac{x}{g_2(x)} = 0 \quad \Leftrightarrow \quad g_2^+ (- \infty) = -\infty.
	\]
	The following implications also follow as in the previous case:
	\[
	g^+_2 (-\infty) = -\infty \ \Rightarrow \ g_1 (- \infty) = 0 \ \Rightarrow \ - \infty\text{ is natural}\ \Rightarrow \ g^+_2 (- \infty) =- \infty.
	\]
	Hence, the equivalence in Lemma \ref{lem: main2} holds for the boundary point \(-\infty\).
	The proof is complete. 
\end{proof}
\begin{proof}[Proof of Lemma \ref{lem: main1} (i)]
	First, assume that \(\X\) is an FD process. Fix \(y \in J^\circ, \alpha > 0\) and let \(g \in C_0 (J)\) be such that \(g (J) \subset [0, 1]\) and \(g (y) = 1\). Furthermore, define
	\[
	R_\alpha g \triangleq \int_0^\infty e^{- \alpha s} T_s g \hspace{0.05cm} ds.
	\]
	It is well-known (\cite[Section III.2.6]{RY}) that \(R_\alpha g \in C_0 (J)\) and that \(e^{- \alpha \cdot} R_\alpha g (\X)\) is a \(\P_x\)-supermartingale for every \(x \in J\).
	Moreover, as \(t \mapsto T_t\) is continuous in the origin, we also see that \(R_\alpha g (y) > 0\).
	The optional stopping theorem yields that 
	\[
	R_\alpha g (x) \geq \E_x \big[ e^{- \alpha \t_y} R_\alpha g (\X_{\t_y}) \1_{\{\t_y < \infty\}}\big] = R_\alpha g (y) \E_x \big[ e^{- \alpha \t_y} \big].
	\]
	As \(R_\alpha g \in C_0(J)\), this inequality implies \eqref{eq: equiv fd}.
	
	Conversely, assume that \eqref{eq: equiv fd} holds. 
	By \cite[Proposition III.2.4]{RY}, \(\X\) is an FD process if and only if \(
	T_t (C_0(J)) \subset C_0(J)
	\)
	for all \(t > 0\). As \(\X\) is a Feller process, we only need to show that \(T_t f\) vanishes at infinity for every \(f \in C_0(J)\) and~\(t > 0\). Of course, for this property we can restrict our attention to open boundaries.
	
	Denote the left boundary point of \(J\) by \(l\) and the right boundary point by \(r\). Let \(g_1\) be as in the proof of Lemma \ref{lem: main2} and assume that \(l\) is open and finite.
	For \(l < x < r\) a little calculus yields that
	\begin{align*}
		\infty > (x - l) g^+_1 (x) + g_1 (l) - g_1 (x) 
		&= \int_l^x \big(g^+_1 (x) - g^+_1 (z)\big) dz
		\\&= \int_l^x \int_{(z, x]} 2 \alpha g_1 (u) \m(du) dz
		\\&\geq 2 \alpha g_1 (l) \int_l^x \m ((z, x]) dz.
	\end{align*}
	As \(l\) is open (i.e. \(\u (r) = \infty\), where \(\u\) is as in \eqref{eq: u func}), this inequality yields \(g_1(l) = 0\).
	Similarly, \(g_2 (r) = 0\) holds in case \(r\) is open and finite. In summary, \eqref{eq: equiv fd} holds for all open boundaries irrespective whether these are finite or infinite.
	
	Take \(f \in C_0(J)\) and \(\alpha, \varepsilon> 0\). 
	For \(l < y < x < r\) we have
	\begin{align}\label{eq: simp ineq}
		\P_x(\X_\alpha < y) \leq \P_x (\t_y < \alpha) \leq e^{\alpha^2} \E_x \big[ e^{- \alpha \t_y} \big].
	\end{align}
	Suppose that the right boundary \(r\) is open. Then, as \(f \in C_0(J)\), there exists a \(z \in J^\circ\) such that \(|f (x)| \leq \varepsilon\) for all \(z \leq x\). Now, taking \eqref{eq: equiv fd} and \eqref{eq: simp ineq} into account, we obtain
	\begin{equation*}\begin{split}
			| T_\alpha f (x) | & \leq \E_x \big[ | f(\X_\alpha) | \1_{\{\X_\alpha \geq z\}} \big] + \E_x \big[ | f(\X_\alpha) | \1_{\{\X_\alpha < z\}} \big] \\ &\leq \varepsilon + \|f\|_\infty \P_x(\X_\alpha < z) \to \varepsilon \text{ as } x \to r.
	\end{split}\end{equation*}
	This implies that \(T_\alpha f (x) \to 0\) as \(x \to r\).
	
	Similarly, when the left boundary \(l\) is open it follows that \(T_\alpha f (x) \to 0\) as \(x \to l\). We conclude that \(T_\alpha f\) vanishes at infinity. The proof is complete.
\end{proof}
\begin{proof}[Proof of Lemma \ref{lem: main1} (ii)]
	By Lemma \ref{lem: main2}, \eqref{eq: equiv mg} holds if and only if all infinite boundary points are natural. Thus, \eqref{eq: equiv mg} holds for the diffusions \(\X\) and \(\X_{\cdot \wedge \i}\) simultaneously. Consequently, we can w.l.o.g. assume that \(\X = \X_{\cdot \wedge \i}\).
	
	Let \(l\) be the left boundary point of \(J\) and let \(r\) be the right boundary point. In case \(- \infty < l < r < + \infty\) the process \(\X= \X_{\cdot \wedge \i}\) is bounded and the claim of Lemma~\ref{lem: main1} (ii) is obvious. Below we distinguish between the cases where \(- \infty < l < r = \infty\) and \(- \infty = l < r = \infty\). The remaining case \(- \infty = l < r < \infty\) is similar to the former.

	If \(\X_t \in L^1(\P_x)\) for all \(t > 0\), then the Markov property yields that 
	\[
	\P_x\text{-a.s.} \quad	\E_x \big[ \X_{t} | \mathcal{F}_{s}\big] = \E_{\X_{s}} \big[ \X_{t - s} \big], \quad 0 \leq s < t.
	\]
	Hence, as martingales always have constant expectation, we have the following:
	\begin{lemma} \label{lem: chara mg}
		\(\X\) is a \(\P_x\)-martingale for all \(x \in J^\circ\) if and only if \(\X_t \in L^1(\P_x)\) and
		\(
		\E_x [\X_{t} ] = x\) for all \(x \in J^\circ\) and \(t > 0\). 
	\end{lemma}
	In the following we prove that the \emph{if} condition from Lemma \ref{lem: chara mg} is equivalent to~\eqref{eq: equiv mg}.\\
	
	\noindent
	\emph{Case 1: \(- \infty < l < r = \infty\).} Fix \(x \in J^\circ = (l, \infty)\) and \(t > 0\). First of all, \(\X_t \in L^1 (\P_x)\) follows from Fatou's lemma as \(\X\) is a local martingale which is bounded from below. 
	For \(x < y < r = \infty\) the stopped process \(\X_{\cdot \wedge \t_y}\) is \(\P_x\)-a.s. bounded and consequently, a \(\P_x\)-martingale. 
	As \(\X_{t} \in L^1(\P_x)\), the dominated convergence theorem yields that 
	\begin{align*}
		\E_x \big[ \X_{t} \big] &= \lim_{y \to \infty}\E_x \big[ \X_{t} \1_{\{\t_y > t\}} \big] 
		\\&= \lim_{y\to\infty} \big(\E_x \big[ \X_{t \wedge \t_y} \big] -  \E_x \big[ \X_{\t_y} \1_{\{\t_y \leq t\}} \big]\big)
		\\&= x - \lim_{y \to \infty} y\P_x (\t_y \leq t).
	\end{align*}
	Thus, by Lemma \ref{lem: chara mg}, \(\X\) is a \(\P_x\)-martingale for all \(x \in J^\circ\) if and only if
	\[
	\lim_{y \to \infty} y \P_x (\t_y \leq t) = 0\] for all \(x \in J^\circ\) and \(t > 0\). Taking this observation into consideration, the next lemma completes the proof of Lemma \ref{lem: main1} (ii) for the current case.
	\begin{lemma} \label{lem: inte alpha t}
		Let \(x \in J^\circ\). Then, 
		\(
		\lim_{y \to \infty} y \P_x (\t_y \leq t) = 0\) for all \(t > 0\) if and only if \(\lim_{y \to \infty} y \E_x [e^{- \alpha \t_y} ] = 0\) for all \(\alpha > 0\).
	\end{lemma}
	\begin{proof} 
		Take \(\alpha > 0\). Fubini's theorem yields that 
		\begin{align*}
			\int_0^\infty e^{- \alpha t} \P_x (\t_y \leq t) dt 
			&= \int \int_u^\infty e^{- \alpha t} dt \hspace{0.03cm} \P_x (\t_y \in du)
			= \tfrac{1}{\alpha} \E_x \big[ e^{- \alpha \t_y} \big].
		\end{align*}
		Furthermore, for every \(y \geq x\) we have 
		\[
		(y - l) \P_x (\t_y \leq t) = \E_x \big[ (\X_{t \wedge \t_y} - l) \1_{\{\t_y \leq t\}} \big] \leq \E_x \big[ \X_{t \wedge \t_y} \big] - l = x - l,
		\]
		which implies
		\(
		|y| \P_x (\t_y \leq t) \leq x - l + |l|.
		\)
		Thus, if \(\lim_{y \to \infty} y\P_x (\t_y \leq t) = 0\) for all \(t > 0\), then the dominated convergence theorem yields
		\[
		\lim_{y \to \infty} y \E_y \big[ e^{- \alpha \t_y} \big] = \lim_{y \to \infty} \int_0^\infty \alpha e^{- \alpha t} y \P_x(\t_y \leq t) dt = 0.
		\]
		This is the \emph{only if} implication. 
		
		Conversely, if \(\lim_{y \to \infty} y \E_x [e^{- \alpha \t_y}] = 0\), then
		\begin{align} \label{eq: one implication} \lim_{y \to \infty} y \P_x(\t_y \leq \alpha) \leq e^{\alpha^2} \lim_{y \to \infty} y \E_x \big[e^{- \alpha \t_y}\big] = 0. \end{align}
		This gives the \emph{if} implication. The proof is complete.
	\end{proof}
	
	\noindent
	\emph{Case 2: \(- \infty = l < r = \infty\).} 
	We start with a version of \cite[Lemma 1]{Kotani2006}:
	\begin{lemma} \label{lem: inte}
		For all \(t > 0\) and \(x \in J = \mathbb{R}\) we have \(\X_t \in L^1(\P_x)\).
	\end{lemma}
	For completeness, we provide a proof for Lemma \ref{lem: inte} at the end of this section.
	Suppose now that \eqref{eq: equiv mg} holds and take \(x \in \mathbb{R}\). 
	As in \eqref{eq: one implication}, we obtain
	\[
	\lim_{y \to \infty} y \P_x (\t_y \leq t) = \lim_{y \to \infty} y \P_x (\t_{- y} \leq t) = 0, \quad t > 0.
	\]
	Now, by virtue of Lemma \ref{lem: inte}, the dominated convergence theorem yields that
	\begin{align*}
		\E_x \big[ \X_{t} \big] &= \lim_{y \to \infty} \E_x \big[ \X_t \1_{\{\t_y \wedge \t_{-y} > t\}}\big]
		\\
		&= \lim_{y \to \infty} \big(\E_x \big[ \X_{t \wedge \t_y \wedge \t_{y-}} \big] - \E_x \big[ \X_{\t_y \wedge \t_{-y}} \1_{\{\t_y \wedge \t_{-y} \leq t\}}\big] \big)
		\\
		&= x - \lim_{y \to \infty} \big(y \P_x (\t_y \leq t, \t_y < \t_{-y}) - y \P_x (\t_{-y} \leq t, \t_{-y} < \t_y)\big) 
		\\&= x
	\end{align*}
	for all \(t > 0\). Hence, the process \(\X\) is a \(\P_x\)-martingale by Lemma \ref{lem: chara mg}.
	
	Conversely, assume that \(\X\) is a \(\P_x\)-martingale for all \(x \in \mathbb{R}\) and take \(a \in \mathbb{R}\). By the optional stopping theorem, the stopped process \(\X_{\cdot \wedge \t_a}\) is a \(\P_x\)-martingale for all \(x \in \mathbb{R}\). For suitable initial values, \(\X_{\cdot \wedge \t_a}\) is a diffusion with state space \([a, \infty)\) (or with state space \((- \infty, a]\)). Note that \(\X_{\cdot \wedge \t_a}\) has the same boundary behavior at~\(\infty\) (or at~\(- \infty\)) as the unstopped process \(\X\), see \cite[Section 3.9, pp. 102 -- 105]{itokean74}. Now, the previous case and Lemma \ref{lem: main2} yield that \(\infty\) and \(- \infty\) are natural. Hence, again by Lemma~\ref{lem: main2}, \eqref{eq: equiv mg} holds and the proof is complete. 
\end{proof}

\begin{proof}[Proof of Lemma \ref{lem: inte}]
	We use a suitable Lyapunov function. Such a function was also used in the proof of \cite[Lemma 1]{Kotani2006}, but it was not given explicitly. Let \(- \infty < a < 0 < b < \infty\) and let \(g \colon \mathbb{R} \to [0, 1]\) be a continuous function such that \(g \equiv 0\) off \([a, b]\) and \(g > 0\) on \([a, b]\).
	Furthermore, define 
	\[
	f (x) \triangleq \begin{cases} \displaystyle \int_0^x \int_{(0, y]} 2g(z) \m(dz) dy, & \text{for } x \geq 0,\\
		\vspace{-0.4cm}
		\\ \displaystyle \int_x^0 \int_{(y, 0]} 2 g(z) \m(dz) dy,& \text{for } x \leq 0.\end{cases}
	\]
	We note that
	\(
	\frac{1}{2} \frac{d}{d \m}\frac{d^+f}{dx} = g,
	\) \(\lim_{x \to \infty} f(x)/x > 0\) and \(\lim_{x \to - \infty} f(x)/ (-x) > 0\).

	Take \(y > (-a) \vee b\).
	As \(\frac{1}{2} \frac{d}{d\m}\frac{d^+}{dx}\) is the generator of the stopped diffusion \(\X_{\cdot \wedge \t_y \wedge \t_{-y}}\) and \(f\) is in its domain (see \cite[Section~2.7]{freedman}), Dynkin's formula (\cite[Lemma 48, p. 119]{freedman}) yields
	\[
	\E_x \big[ f(\X_{t \wedge \t_y \wedge \t_{-y}})\big] = f (x) + \E_x \Big[ \int_0^{t \wedge \t_y \wedge \t_{-y}} g (\X_s) ds \Big] \leq f (x) + t \|g\|_\infty.
	\]
	Finally, letting \(y \to \infty\) and using Fatou's lemma yields that \(f(\X_t) \in L^1 (\P_x)\). As \(\lim_{x \to \infty} f(x)/x > 0\) and \(\lim_{x \to - \infty} f(x)/ (-x) > 0\), this implies \(\X_t \in L^1(\P_x)\) and the proof is complete.
\end{proof}

\bibliographystyle{plain}

\begin{thebibliography}{1}
	
	\bibitem{Azencott1974}
	R.~Azencott.
	\newblock Behavior of diffusion semi-groups at infinity.
	\newblock {\em Bulletin de la Soci{\'e}t{\'e} Math{\'e}matique de France},
	102:193--240, 1974.
	
	\bibitem{breiman1968probability}
	L.~Breiman.
	\newblock {\em Probability}.
	\newblock Classics in Applied Mathematics. Society for Industrial and Applied
	Mathematics, 1992.
	
	\bibitem{bay2010}
	E.~Bayraktar and H.~Xing.
	\newblock On the uniqueness of classical solutions of Cauchy problems.
	\newblock {\em Proceedings of the American Mathematical Society}, 138(6):2061--2064, 2010.
	
	\bibitem{cet18}
	U.~\c{C}etin.
	\newblock Diffusion transformations, Black--Scholes equation and optimal stopping.
	\newblock The Annals of Applied Probability, 28(5):3102--3151, 2018. 
	
	\bibitem{cet20}
	U.~\c{C}etin and K.~Larson.
	\newblock Uniqueness in Cauchy problems for diffusive real-valued strict local martingales.
	\newblock arXiv:2007.15041v1, 2020. 
	
	\bibitem{chung}
	K.~L.~Chung and J.~B.~Walsh.
	\newblock {\em Markov processes, Brownian motion, and time symmetry}.
	\newblock Springer Science+Business Media, 2nd edition, 2005.
	
	\bibitem{clem86}
	P.~Cl\'ement and C.~A. Timmermans.
	\newblock On \(C_0\)-semigroups generated by differential operators satisfying Ventcel's boundary conditions.
	\newblock {\em Indagationes Mathematicae (Proceedings)}, 89(4):379--387, 1986.
	
	\bibitem{criens2020}
	D.~Criens.
	\newblock Lyapunov criteria for the Feller–Dynkin property of martingale problems.
	\newblock {\em Stochastic Processes and their Applications}, 130(5):2693--2736, 2020.
	
	\bibitem{Delbaen2002}
	F.~Delbaen and H.~Shirakawa.
	\newblock No arbitrage condition for positive diffusion price processes.
	\newblock {\em Asia-Pacific Financial Markets}, 9(3):159--168, 2002.
	
	\bibitem{eberle}
	A.~Eberle.
	\newblock Uniqueness and Non-Uniqueness of Semigroups Generated by Singular Diffusion Operators.
	\newblock Springer Berlin Heidelberg, 1999.
	
	\bibitem{EK}
	S.~N.~Ethier and T.~G.~Kurtz.
	\newblock Markov Processes Characterization and Convergence.
	\newblock Wiley, 2005.
	
	\bibitem{Feller}
	W.~Feller.
	\newblock The parabolic differential equations and the associated semi-group of transformations.
	\newblock {\em Annals of Mathematics, Second Series}, 55(3):468--519, 1952.
	
	\bibitem{freedman}
	D.~Freedman.
	\newblock {\em Brownian Motion and Diffusion}.
	\newblock Springer New York Heidelberg Berlin, 1983.
	
	\bibitem{Hulley2009}
	H.~Hulley.
	\newblock Strict local martingales in continuous financial market models.
	\newblock Ph.D. Thesis, University of Sydney, Sydney, Australia, 2009.
	
	\bibitem{itokean74}
	K.~It\^o and H.~P. McKean, Jr..
	\newblock {\em Diffusion Processes and their Sample Paths}.
	\newblock Springer Berlin Heidelberg, 1974.
	
	\bibitem{kallenberg}
	O.~Kallenberg.
	\newblock {\em Foundations of Modern Probability}.
	\newblock Springer Nature Switzerland, 3rd edition, 2021.
	
	\bibitem{KaraShre}
	I.~Karatzas and S.~E.~Shreve.
	\newblock {\em Brownian Motion and Stochastic Calculus}.
	\newblock Springer New York, 2nd edition, 1991.
	
	\bibitem{knight}
	F.~B.~Knight.
	\newblock {\em Essentials of Brownian Motion and Diffusion}.
	\newblock American Mathematical Society, 1981.
	
	\bibitem{Kotani2006}
	S.~Kotani.
	\newblock On a condition that one-dimensional diffusion processes are
	martingales.
	\newblock In {\em In Memoriam Paul-Andr{\'e} Meyer: S{\'e}minaire de
		Probabilit{\'e}s XXXIX}, 149--156. Springer, 2006.
	
	\bibitem{mandl}
	P.~Mandl. 
	\newblock Analytic Treatment of One-dimensional Markov Processes.
	\newblock Springer Berlin Heidelberg New York, 1968.
	
	\bibitem{pazy}
	A.~Pazy.
	\newblock {\em Semigroups of Linear Operators and Applications to Partial Differential Equations.}
	\newblock Springer New York, 1983.
	
	\bibitem{portstone}
	S.~C.~Port and C.~J.~Stone.
	\newblock {\em Brownian Motion and Classical Potential Theory.}
	\newblock Academic Press, 1978.
	
	\bibitem{RY}
	D.~Revuz and M.~Yor.
	\newblock {\em Continuous Martingales and Brownian Motion}.
	\newblock Springer, Berlin Heidelberg, 3rd edition, 1999.
	
	\bibitem{RW1}
	L.~C.~G.~Rogers and D.~Williams.
	\newblock {\em Diffusions, Markov Processes, and Martingales. Volume 1 - Foundations}.
	\newblock John Wiley \& Sons, 2nd edition, 1994.
	
	\bibitem{RW2}
	L.~C.~G.~Rogers and D.~Williams.
	\newblock {\em Diffusions, Markov Processes, and Martingales. Volume 2 - It\^o Calculus}.
	\newblock Cambridge University Press, 2nd edition, 2000.
	
	\bibitem{urusov17}
	M.~Urusov and M.~Zervos.
	\newblock Necessary and sufficient conditions for the \(r\)-excessive local martingales to be martingales.
	\newblock {\em Electronic Communications in Probability}, 22(10):1--6, 2017.
	
	\bibitem{Vyborny89}
	R.~Vyborny and R.~Nester.
	\newblock L'H\^opital's rule, a counterexample.
	\newblock {\em Elemente der Mathematik}, 44:116--121, 1989.
	
\end{thebibliography}

\end{document}